\documentclass[12pt, a4paper]{amsart}

\usepackage{a4wide}
\usepackage[english]{babel}
\usepackage[T2A]{fontenc}
\usepackage[utf8]{inputenc} 
\usepackage{amsfonts}
\usepackage{amssymb, amsthm, amscd}
\usepackage{amsmath}
\usepackage{mathtools}
\usepackage{needspace}
\usepackage{etoolbox}
\usepackage{lipsum}
\usepackage{comment}
\usepackage{cmap}
\usepackage[pdftex]{graphicx}
\usepackage[unicode]{hyperref}
\usepackage[matrix,arrow,curve]{xy}
\usepackage[usenames,dvipsnames]{xcolor}
\usepackage{colortbl}
\usepackage{textcomp}
\usepackage{cite}
\usepackage{euscript}

\makeatletter
\def\@settitle{\begin{center}
		\baselineskip14\p@\relax
		\bfseries
		\LARGE
		\@title
	\end{center}
}
\makeatother

\newtheorem{theorem}{Theorem}
\newtheorem{lemma}{Lemma}

\theoremstyle{definition}
\newtheorem{definition}{Definition}

\renewcommand{\geq}{\geqslant}
\renewcommand{\leq}{\leqslant}

\newcommand{\Q}{\Omega}

\title{On generating sets of infinite symmetric group}
\author{Andrei V. Semenov, Aleksandra Denisova}

\thanks{First author was supported by Young Russian Mathematics award and by ``Native Towns'', a social investment program of PJSC ``Gazprom Neft''}

\address{Andrei V. Semenov:
Chebyshev Laboratory, 
St. Petersburg State University, 14th Line V.O., 29B, 
Saint Petersburg 199178 Russia}

\email{asemenov.spb.56@gmail.com}

\begin{document}
\maketitle

\begin{abstract}
It was shown that in a group of bijections of an infinite set some families of subsets, related to the cardinality of some eigenspaces, are generating. Besides, we derived a criterion for generating by sets of this kind.
\end{abstract}
\vspace{10mm}

\section{Introduction} The symmetric group $S(\Q)$ (for the arbitrary infinite set $\Q$) is a "large" group, which means that it is difficult to study even particular cases. Thus, hardly anything is known about the generating sets of this group. George Bergman studied generating sets in some special cases (see lemmas 1-4 in \cite{1}) and made a few abstract conclusions about this group and its generating sets (see also Theorem 5 in \cite{1}).
Manfred Droste studied classes of words and chains of subgroups in $S(\Q)$ in his papers \cite{6, 7, 8}. Unfortunately, most of these results are too abstract to be applied to particular problems. \par
The situation is different for the classification of simplier objects: for example, all normal subgroups $S(\Q)$ have been classified (see \S 8 in \cite{2}). Also, the majority of results about the "classical" objects of the group theory have been delivered (see \cite{3} and \cite{4}).  \par
This paper considers the problem of generating certain infinite symmetric groups $S(\Q)$ for the arbitrary infinite set $\Q$. It is shown that in an infinite group of bijections some families of subsets, connected with the number of proper eigenspaces and their structure, are generating. We derived a criterion of generating for subsets of this kind, therefore solving a practical problem of generating in $S(\Q)$ for families of subsets of this structure.

\section{Definitions and established facts}
In this paper we fix the infinite set $\Q$, its group of bijection $S(\Q)$ and its cardinality $|\Q|$. For obvious reasons, $\Q$ stands for the set and for its cardinality. Also, $I_n$ denotes a set of elements of order $n$ in $S(\Q)$.

\begin{definition} The subset $ U \subset \Q$ is identified as {\it f-eigenspace } if $f(x) \in U$ for all $x\in U$. 
\end{definition}

\begin{definition} Let $ M_f$ denote a set of all non-stable elements of $f \in S(\Q)$.
\end{definition}

\begin{definition}
Let $W_{\alpha, \beta} (\Q)$ denote a set of all permutations $f \in S(\Q)$, such that $ M_f$ can be expressed as a disjoint union of at most $\alpha$ $f$-eigenspaces, such that the cardinality of each eigenspace is less than (or equal to) $\beta$.
\end{definition}

\begin{definition}
Let $K_{\alpha, \beta} (\Q)$ denote a set of all permutations $f \in S(\Q)$, such that $ M_f$ can be expressed as a disjoint union of $\alpha$ $f$-eigenspaces, such that the cardinality of each eigenspace is less than (or equal to) $\beta$.
\end{definition}

\begin{definition}
Let $R_{\alpha, \beta} (\Q)$ denote a set of all permutations $f \in S(\Q)$, such that $ M_f$ can be expressed as a disjoint union of at most $\alpha$ $f$-eigenspaces, such that the cardinality of each eigenspace is equal to $\beta$.
\end{definition}

\begin{definition}
Let $S_{\alpha, \beta} (\Q)$ denote a set of all permutations $f \in S(\Q)$, such that $ M_f$ can be expressed as a disjoint union of $\alpha$ $f$-eigenspaces, such that the cardinality of each eigenspace is equal to $\beta$.
\end{definition}

\begin{theorem}[\cite{1}, Theorem 5]
Let $\Q$ be an infinite set, and consider a chain of subgroups $\{ G_i\}_{i \in I} \leq S(\Q)$. If $|I| \leq |\Q|$ and $\bigcup_{i \in I} G_i = S(\Q)$, then there exists such an index $i$ that $G_j = S(\Q)$ for any $j\geq i $.
\end{theorem}

\begin{theorem}[Schreier-Ulam, see \cite{2}] For any normal subgroup $H \leq S(\Q)$ there exists such a cardinal $\alpha$, such that $H = S_\alpha(\Q)=\{f\in S(\Q) \mid |M_f| \leq\alpha \leq \Q\}$.
\end{theorem}

\section{Study of $W_{\alpha, \beta} (\Q)$}
The goal of studying minimality for generating sets cannot be reached because of this almost trivial result:
\begin{theorem}
There is no minimal generating subset for $S(\Q)$.
\end{theorem}
\begin{proof}
Suppose the contrary: there exists such a set $S$ which is at least countable. Then there exists a countable subset $A=\{a_1, a_2, a_3, \dots\}$, so we can take $S_0= S \setminus A, S_1=S_0\cup \{a_1\}$ and define $S_n := S_{n-1}\cup\{a_n\}$. \par 
Now we can notice that $\langle S_0\rangle \subseteq \langle S_1\rangle \subseteq \langle S_2\rangle \subseteq\dots\subseteq S(\Q)$, and any set from this chain is not equal to the group. But $\bigcup\limits_{i=0}^{\infty} S_i=S(\Q)$, which contradicts the previous Theorem. Hence there is no minimal generating subset for $S(\Q)$.
\end{proof}

Nonetheless, one can study almost arbitrary constructions, linked with the orbits of acting $S(\Q)$ on $\Q$. Let us describe the "largest" case $W_{\alpha, \beta} (\Q)$.

\begin{theorem}
For any cardinals $\alpha, \beta$, if $\alpha \cdot \beta < \Q$, then $W_{\alpha, \beta} (\Q)$ is not generating for $S(\Q)$.
\end{theorem}
\begin{proof}
Fix $\alpha, \beta$ --- arbitrary cardinals, smaller than $\Q$. Without loss of generality we can assume that $\alpha > \beta$. Fix the permutation $f\in W_{\alpha, \beta}(\Q)$. $M_f$ could be obtained as a union of less than $\alpha$ $f$-eigenspaces, such that the cardinality of each eigenspace is smaller than $\beta$. Also, the cardinality of $M_f$ is less than $\alpha\cdot\beta=\alpha$, hence $f\in S_{\alpha}(\Q)$. So $W_{\alpha, \beta}(\Q)\subseteq S_{\alpha}(\Q)$, which is not a generating set, because $\alpha<\Q$.
\end{proof}

\begin{lemma}
$W_{\Q, \alpha}(\Q)$ is a subset of $W_{\Q, \beta}(\Q)$ for any cardinals $\alpha < \beta < \Q$.
\end{lemma}
\begin{proof}
Fix $\alpha<\beta$ and $f\in W_{\Q, \alpha}(\Q)$. The cardinality of any $f$-eigenspace is not greater than $\alpha$, so it is not greater than $\beta$. Hense $f\in W_{\Q, \beta}(\Q)$.
\end{proof}

\begin{lemma}
$W_{\beta, \Q} (\Q) \subseteq W_{\alpha, \Q}(\Q)$ for any cardinals $\alpha < \beta$.
\end{lemma}
\begin{proof}
Fix $\alpha < \beta$ and $f\in W_{\beta, \Q}(\Q)$. Obviously, $M_f=\bigcup\limits_{\leq \beta} V_f$, where $V_f$ are $f$-eigenspaces, such that the cardinality of each eigenspace is not greater than $\Q$. Since $V = \bigcup\limits_{\leq \beta} V_f$ is an $f$-eigenspace itself,  $f\in W_{\alpha, \Q}(\Q)$, because $1 \leq \alpha$ and $|V| < \Q$.
\end{proof}

\begin{theorem}[Criterion for generating sets for $W_{\alpha,\beta} (\Q)$]
Set $W_{\alpha,\beta} (\Q)$ is generating for $S(\Q)$ iff at least one of the cardinals $\alpha, \beta$ is equal to $\Q$.
\end{theorem}
\begin{proof}
In order to prove the first statement, let $W_{\alpha, \beta}(\Q)$ generate $S(\Q)$ and let us assume that $\alpha<\Q$ and $\beta<\Q$. So by theorem 4 $|S(\Q)| = |W_{\alpha, \beta} (\Q)| < |S(\Q)|$, which contradicts our assumption.\par 
On the other hand, let $\alpha=\Q$ or $\beta=\Q$. Then, following Lemmas 1 and 2, $I_2 = W_{\Q, 2} \subseteq W_{\alpha, \beta}$, and, as we show in Theorem 8, $I_2$ is a generating set. So
$$S(\Q) = \langle I_2 \rangle \subseteq \langle W_{\alpha, \beta}(\Q) \rangle \subseteq S(\Q),$$
and any of $W_{\alpha,\beta} (\Q)$ is generating.
\end{proof}

\section{Study of $S_{\alpha, \beta} (\Q)$}
Let us develop the case of the smallest of such sets, as described in chapter 1.

\begin{lemma}
$S_{\beta, \Q} (\Q) \subseteq S_{\alpha, \Q}(\Q)$ for any cardinals $\alpha < \beta$.
\end{lemma}
\begin{proof}
Let $\alpha < \beta$ be cardinals and fix $f\in S_{\beta, \Q}(\Q)$. It is easy to see that $M_f=\bigcup\limits_{\beta} V_f$, where $V_f$ are $f$-eigenspaces with the cardinality equal to $\Q$. Because  $\alpha<\beta$, there exists $\gamma < \Q$, such that $\beta=\alpha+\sum\limits_{\gamma} \alpha$. So $M_f$ could be obtained as $(\bigcup\limits_{\alpha} V_f) \cup (\bigcup\limits_{\sum\limits_{\gamma}\alpha} V_f)$, i. e. as a union of $\alpha$ $f-$eigenspaces, and since 
$| \bigcup\limits_{\sum\limits_{\gamma}\alpha} V_f | = \sum \limits_{\gamma} \alpha \Q = \Q, $ it follows that $f\in S_{\alpha, \Q}(\Q)$.
\end{proof}

\begin{lemma}
$S_{\Q, \alpha}(\Q)\subseteq S_{\Q, \beta}(\Q)$ for any infinite cardinals $\alpha < \beta$.
\end{lemma}
\begin{proof}
Fix $\alpha<\beta$ and $f\in S_{\Q, \alpha}(\Q)$. Obviously $M_f=\bigcup\limits_{\Q} V_f$, where $V_f$ are $f$-eigenspaces, such that the cardinality of each eigenspace is equal to $\alpha$. Observe that $\Q$ could be obtained as $\alpha \cdot \beta \cdot \Q$, so $M_f=\bigcup\limits_{\Q}(\bigcup\limits_{\beta} V_f)$, where the cardinality of any $V_f$ is equal to $\alpha$ and $\bigcup\limits_{\beta} V_f$ is an $f$-eigenspace, whose cardinality is equal to $\beta$. Hence, $f\in S_{\Q, \beta}(\Q)$ by definition.
\end{proof}

\begin{theorem}
$\langle I_n \rangle \subseteq \langle S_{\Q, n}(\Q) \rangle$ for any $n \geq 2$.
\end{theorem}
\begin{proof}
It is sufficient to prove that any permutation of order $n$ could be obtained as a finite composition of permutations from $S_{\Q, n}(\Q)$. \par 
Fix $n\in\mathbb{N}$ and $f\in I_n$. Consider two cases: when $|M_f|=\Q$ and when $|M_f|<\Q$.\par 
The first case is obvious by definition (because $M_f$ could be obtained as $\Q$ $f-$eigenspaces, such that the cardinality of each eigenspace is equal to $n$).\par
In the second case we can take such $g_1\in S_{\Q, n}$, that  $g_{1_{|_{M_f}}}=f$. Observe that $|M_{g_1}\setminus M_f|=\Q$, because $|M_f|<\Q$, so $M_{g_1}\setminus M_f$ could be obtained as a disjoint union of $f$-eigenspaces, such that the cardinality of each eigenspace is equal to $n$. One can construct a function $g_2$, such that $g_{2_{|_{M_{g_1}}\setminus M_f}}=g_1^{-1}$ and $g_{2_{|_{M_f}}}= id_{\Q}$. It is obvious that $g_2$ also lies in $S_{\Q, n}(\Q)$, so the composition of $g_1$ and $g_2$ equals to $f$ and $f\in \langle S_{\Q, n} (\Q) \rangle$.
\end{proof}

\begin{lemma}
$S_{\Q, 2}(\Q) \subseteq S_{\Q, \alpha}(\Q)$ for any infinite cardinal $\alpha$.
\end{lemma}
\begin{proof}
Let $f\in S_{\Q, 2}(\Q)$, so $M_f=\bigcup\limits_{\Q} V_f$, where $V_f$ are $f$-eigenspaces, such that the cardinality of each eigenspace is equal to $2$. $\Q$ could be obtained as $\Q \cdot (\alpha \cdot 2)$, because $2< \alpha< \Q$. Then $M_f=\bigcup\limits_{\Q}(\bigcup\limits_{\alpha} V_f)$, where the cardinality of any $V_f$ is equal to $2$ and $\bigcup\limits_{\alpha} V_f$ is an $f$-eigenspace, such that the cardinality of each eigenspace is equal to $\alpha$. Hence, $f\in S_{\Q, \alpha}(\Q)$.
\end{proof}

\begin{theorem}[Criterion for generating sets for $S_{\alpha,\beta} (\Q)$]
Set $S_{\alpha,\beta} (\Q)$ is generating for $S(\Q)$ iff at least one of the cardinals $\alpha, \beta$ is equal to $\Q$.
\end{theorem}
\begin{proof}
The first item follows from Theorem 4, because if $a \cdot \beta < \Q$ then $S_{\alpha, \beta} (\Q) $ can't be a generating set.\par 

Following Lemma 5 and Theorem 6, $\langle I_2 \rangle \subseteq \langle S_{\Q, 2}(\Q) \rangle \subseteq \langle S_{\Q, \beta}(\Q) \rangle$ for any infinite cardinal $\beta$, and, as we show in Theorem 8, $I_2$ is a generating set. Then, for any infinite $\beta$ $S_{\Q, \beta} (\Q)$ is a generating set. So, by Lemma 4, $S_{\alpha, \Q}(\Q)$ will be generating for any cardinal $\alpha$, because $S_{\Q, \Q}(\Q) \subseteq S_{\alpha, \Q}(\Q)$. \par

Finally, if $\beta$ is a natural number, by Theorem 6, $\langle I_{\beta} \rangle \subseteq \langle S_{\Q, \beta}(\Q) \rangle$ and, as we show in Theorem 8, $I_{\beta}$ is a generating set.
\end{proof}

\section{The main result}

\begin{theorem}
$\langle I_n \rangle = S(\Q)$ for any $n\in\mathbb{N}$.
\end{theorem}
\begin{proof}
It is easy to see that the normal closure of $I_n$ coincides with the subgroup, generated by $I_n$:
$$\bar{I}_n = \langle I_n^{S(\Q)} \rangle = \langle \{gsg^{-1} \mid s \in I_n\} \rangle = \langle I_n \rangle.$$
Assume $\langle I_n \rangle < S(\Q)$. So, by Schrier-Ulam Theorem, there exists such a cardinal $\alpha <\Q$, that $\langle I_n\rangle = S_{\alpha}(\Q)$, but it contradicts the fact that there exists such a permutation $f\in I_n$, that $M_f=\Q$, which is impossible by definition.
\end{proof}

\begin{theorem} The following statements hold:
\begin{enumerate} 
    \item $W_{\alpha,\beta} (\Q), R_{\alpha,\beta} (\Q), K_{\alpha,\beta} (\Q), S_{\alpha,\beta} (\Q)$ are not generating sets for any $\alpha\cdot \beta < \Q$.
    \item  $W_{\alpha,\Q} (\Q), R_{\alpha,\Q} (\Q), K_{\alpha,\Q} (\Q), S_{\alpha,\Q} (\Q)$ are generating sets for any $\alpha \leq \Q$.
    \item $W_{\Q,\beta} (\Q), R_{\Q,\beta} (\Q), K_{\Q,\beta} (\Q), S_{\Q,\beta} (\Q)$ are generating sets for any $\beta \leq \Q$.
\end{enumerate}
\end{theorem}
\begin{proof}
Fix $\alpha, \beta$. Observe that by definition
$$S_{\alpha, \beta}(\Q) \subseteq K_{\alpha, \beta}(\Q) \subseteq W_{\alpha, \beta}(\Q) \text{ and } S_{\alpha, \beta}(\Q) \subseteq R_{\alpha, \beta}(\Q) \subseteq W_{\alpha, \beta}(\Q).$$ Then the first case follows from Theorem 4 and the second and the third cases follow from Theorem 7.
\end{proof}

\section{Partial case of $\Q = \mathbb{Z}$}
The case of a countable set is especially interesting, because here we can deal with permutations which act on $\Q$ as {\it cycles}. It is, obviously, stronger than our previous definition and this case requires another study. We consider orbits (cycles) $\mathcal{O}_f(x) = \{f^k (x) \mid k \in \mathbb{Z}\}$, so we obtain three such sets:
\begin{enumerate}
    \item A set of all local finite permutations $LF = \{ f \in S(\Q) \mid \text{for all } x \in \Q \text{ cycle } \mathcal{O}_f(x) \text{ is finite}\}$. As an example one could take a permutation  $..(12)(34)(56)..$
    \item A set of all ringed permutations $R = \{f \in S(\Q) \mid \text{the set of all the orbits for } f \text{ is finite} \}$. As an example one could take a permutation $f(n) = n+1$.
    \item A set of all wild permutations  $W = \{f \in S(\Q) \mid \text{There exists infinitely many countable cycles } \mathcal{O}_f(x)\}$. In order to take this permutation one could view $\mathbb{Z}$ as an infinite union of a countable set.
\end{enumerate}

\begin{lemma}
$R\subseteq \langle LF \rangle$.
\end{lemma}
\begin{proof} It suffices to prove that we can obtain an infinite cycle by the composition of local finite functions, so without loss of generality we can assume that $f \in R$ consists of only one infinite cycle, so
$$f=(\dots f^{-2}(x_0), f^{-1}(x_0), x_0, f(x_0), f^2(x_0), \dots)$$ 
Consider two local finite functions $f_1,f_2:\Q \longrightarrow \Q$, such that
$$f_1(f^n(x_0))=f^{-n-1}(x_0) \text{ and } f_2(f^n(x_0))=f^{-n}(x_0) \text{ for an integer }n,$$
and let us define them as identities on the complement of this cycle. \par 
In order to prove the statement, we need to show that $f^n(x_0)$ goes to $f^{n+1}(x_0)$ by $f_1f_2$ for all $n \in \mathbb{Z}$. But it is easy to see that
$$f_2(f_1(f^n(x_0)))=f_2(f^{-n-1}(x_0))=f^{n+1}(x_0).$$ 
Hence, $f = f_1f_2 \in \langle LF \rangle$.
\end{proof}

\begin{lemma}
$I_2\subseteq \langle R \rangle$.
\end{lemma}
\begin{proof}
Fix $f\in I_2$ and consider two cases: when $f$ is equal to a union of  infinitely many transpositions and when $f$ is equal to a union of  finitely many transpositions.\par 
In the first case $f$ can be represented by: 
$$f=(x_1^1, x_1^2)(x_2^1, x_2^2)(x_3^1, x_3^2)(x_4^1, x_4^2)\dots$$ We can assume that $f$ does not contain a cycle of order 1: otherwise one can add a set of all the fixed elements of $f$ as a cycle to $g_1$ (this function will be described later), and add an inversed cycle to $g_2$. Let us now define two ringed functions $g_1, g_2$:
$$g_1=(\dots x_4^2, x_3^1, x_2^2, x_1^1, x_2^1, x_1^2, x_4^1, x_3^2, \dots),$$
$$g_2=(\dots x_3^2, x_4^2, x_1^2, x_2^2, x_1^1,x_2^1, x_3^1, x_4^1, \dots).$$
One can check that $g_1g_2=f$ by straightforward calculations.\par
In order to obtain $f$ in the second case we only need to construct such a function $g_1$ which consists of an infinite number of transpositions, such that the set of orbits of $f$ lies in the set of orbits of $g_1$, and also we need to construct a function $g_2$, such that $M_{g_1}\setminus M_{g_2}=M_f$ and $M_{g_2}\bigcup M_f=M_{g_1}$. Obviously, the composition $g_1g_2$ is equal to $f$, hence $f$ could be obtained as a finite composition of ringed functions.
\end{proof}

\begin{lemma}
$LF\subseteq \langle W\rangle$.
\end{lemma}
\begin{proof}
Fix $f\in LF$, and let $f$ be not finite. We need to prove that $f\in\langle W\rangle$.\par
Consider two cases: when $\Q \setminus M_f$ is finite and when $\Q \setminus M_f$ is infinite.
In the first case $f$ could be obtained as a composition of four wild permutations. In order to prove it, we only need to understand how we could obtain a finite cycle $(x_1, x_2, \dots , x_n)$. Let us take two infinite cycles: the cycle
$$(\dots a_{-1}, a_0,, x_1, x_2, \dots, x_n, a_1, a_2, \dots)$$ 
and the cycle
$$(\dots, a_2, a_1, x_n, a_0, a_{-1}, \dots),$$ 
where $a_i$ are some elements out of the cycle, and it is obvious that the number of such elements is infinite. Let us count the cycles of $f$, whose length is greater than $2$ by integers (we can do it anyway, because all of the orbits are finite, so the number of cycles is infinite, and the number of cycles with a cardinality greater than 2 is also infinite, because $f$ is not finite). One needs to construct functions $g_1, g_2$ by such rule: let us view a set of all elements from even cycles (it is, obviously, an infinite set) as a disjoint infinite union of infinite subsets. For any odd cycle $(x_1, x_2, \dots , x_n)$ we fix only one of these (infinite) sets and take $a_i$-s from this set. Functions $g_1$ and $g_2$ consist of cycles $(\dots a_{-1}, a_0,, x_1, x_2, \dots, x_n, a_1, a_2, \dots)$ and $(\dots, a_2, a_1, x_n, a_0, a_{-1}, \dots)$ accordingly. So we "eliminated" all odd cycles.\\
Then one needs to construct functions $h_1, h_2$, by the same rule: we need to "eliminate" all even cycles, using elements from odd cycles as $a_i$-s. All such functions are wild, because they consist of infinitely many cycles (one cycle for one orbit, and the number of orbits is infinite), and all such cycles are countable by construction. So $h_i \in W$ and $g_i \in W$ and
$$f = h_1h_2 g_1 g_2.$$
In the second case one can obtain $f$ by the composition of two wild functions: because $\Q \setminus M_f$ is infinite, we could take $a_i$-s from this set.\par
Finally, in the exceptional case, when $f$ is finite, this algorithm gives us only ringed functions (because there are only finitely many cycles of length greater than 2). But one can easily make wild functions from them: we only need to add infinitely many cycles from the elements of $\Q \setminus M_f$ to $g_1$ and all the cycles inverse to them to $g_2$.

\end{proof}

\begin{theorem} Let $\Q$ be a countable set. Then any of the sets $LF$, $R$ and $W$ is a generating set for $S(\Omega)$.
\end{theorem}
\begin{proof}
As we know from the previous Lemmas, $\langle I_2 \rangle \subseteq \langle R \rangle \subseteq \langle LF \rangle \subseteq  \langle W \rangle \subseteq S(\Q)$, and $\langle I_2 \rangle = S(\Q)$ by Theorem 8, whence the result follows.
\end{proof}

\end{document}